\documentclass[12pt,
psamsfonts]{amsart}

\usepackage{amssymb,amsfonts,amsmath}
\usepackage[all,arc]{xy}
\usepackage{enumerate}
\usepackage{mathrsfs}
\usepackage{fullpage}
\usepackage{xspace}
\usepackage[margin=1.0in]{geometry}
\usepackage{tcolorbox}
\usepackage{tikz-cd}
\usepackage{color}
\usepackage{aliascnt}
\usepackage[foot]{amsaddr}
\usepackage{hyperref}

\newtheorem{thm}{Theorem}[section]

\newaliascnt{theo}{thm}
\newtheorem{theo}[theo]{Theorem}
\aliascntresetthe{theo}

\newaliascnt{cor}{thm}
\newtheorem{cor}[cor]{Corollary}
\aliascntresetthe{cor}

\newaliascnt{prop}{thm}
\newtheorem{prop}[prop]{Proposition}
\aliascntresetthe{prop}

\newaliascnt{defn}{thm}
\newtheorem{defn}[defn]{Definition}
\aliascntresetthe{defn}

\newaliascnt{lem}{thm}
\newtheorem{lem}[lem]{Lemma}
\aliascntresetthe{lem}

\newaliascnt{conj}{thm}
\newtheorem{conj}[conj]{Conjecture}
\aliascntresetthe{conj}

\newaliascnt{que}{thm}

\aliascntresetthe{que}

\newaliascnt{ass}{thm}

\aliascntresetthe{ass}

\theoremstyle{remark}
\newaliascnt{rem}{thm}
\newtheorem{rem}[rem]{Remark}
\aliascntresetthe{rem}

\theoremstyle{definition}

\newtheorem{notn}[thm]{Notation}

\newcommand{\Z}{\mathbb{Z}\xspace}
\newcommand{\Q}{\mathbb{Q}\xspace}
\newcommand{\G}{\mathbb{G}\xspace}
\DeclareMathOperator{\Spec}{Spec}
\DeclareMathOperator{\res}{res}
\DeclareMathOperator{\Tr}{Tr}
\DeclareMathOperator{\ord}{ord}

\DeclareMathOperator{\dv}{div}
\DeclareMathOperator{\alb}{alb}

\DeclareMathOperator{\Hom}{Hom}

\DeclareMathOperator{\Alb}{Alb}
\DeclareMathOperator{\Tor}{Tor}
\DeclareMathOperator{\sing}{sing}
\DeclareMathOperator{\geo}{geo}
\DeclareMathOperator{\DM}{DM}
\DeclareMathOperator{\eff}{eff}
\makeatletter
\let\c@equation\c@thm
\makeatother
\numberwithin{equation}{section}

\title{Some results about zero-cycles on abelian and semi-abelian varieties}

\author[*]{Evangelia Gazaki*} \address[*]{\normalfont Department of Mathematics, University of Michigan, 3823 East Hall, 530 Church St., Ann Arbor, MI, 48109, USA. Email: \texttt{gazaki@umich.edu}}
\begin{document}
\maketitle
\begin{abstract} In this short note we extend some results obtained in \cite{Gazaki2015}. First, we prove that for an abelian variety $A$ with good ordinary reduction over a finite extension of $\Q_p$ with $p$ an odd prime, the Albanese kernel of $A$ is the direct sum of its maximal divisible subgroup and a torsion group. Second, for a semi-abelian variety $G$ over a perfect field $k$, we construct a decreasing integral filtration $\{F^r\}_{r\geq 0}$ of Suslin's singular homology group, $H_0^{\sing}(G)$, such that the successive quotients are isomorphic to a certain Somekawa K-group.  

\end{abstract}

\section{Introduction}

For a smooth projective and geometrically integral variety $X$ over a field $k$, the group $CH_0(X)$ of zero-cycles modulo rational equivalence has a filtration \[CH_0(X)\supset A_0(X)\supset T(X)\supset 0,\] where 
$A_0(X)=\ker(CH_0(X)\xrightarrow{\deg}\Z)$ is the kernel of the degree map, while
\[T(X)=\ker(A_0(X)\xrightarrow{\alb_X}\Alb_X(k))\] is the kernel of the Albanese map, that is the higher dimensional analogue of the Abel-Jacobi map of curves. The Albanese kernel, $T(X)$, is the most mysterious part of $CH_0(X)$ and is in general very hard to compute.  Depending on the nature of the base field $k$ there are numerous conjectures concerning the structure of the Albanese kernel. For example, when $k$ is an algebraic number field, the famous Bloch-Beilinson conjectures (\cite{Bloch2000}, \cite{Beilinson1984}) predict that $T(X)$ is a finite group. 

The Beilinson conjectures in particular lie on the deep philosophy that for a smooth projective variety $X$, all Chow groups, $CH_i(X)$, should have an integral filtration arising from a spectral sequence in the conjectural category of mixed motives. 

In \cite{Gazaki2015} we constructed a candidate for such an integral filtration $\{F^r\}_{r\geq 0}$ for the Chow group $CH_0(A)$ of zero-cycles on an abelian variety $A$ over a field $k$. This decreasing filtration has the property that for every $r\geq 0$, there is an isomorphism 
\begin{eqnarray}\label{isoab}F^r/F^{r+1}\otimes\Z\left[\frac{1}{r!}\right]\simeq S_r(k;A)\otimes\Z\left[\frac{1}{r!}\right],\end{eqnarray} where $S_r(k;A)$ is the symmetric quotient of the Somekawa K-group $K(k;A,\ldots,A)$ attached to $r$ copies of $A$. In many cases, for example when the base field is algebraically closed, the $\otimes\Z[1/r!]$ can be omitted in the above isomorphism. We note that after $\otimes\Q$ this filtration coincides with the motivic filtration previously studied by Bloch (\cite{Bloch1976}) and Beauville (\cite{Beauville1986}) (see \autoref{BlochBeauville} for more details).

 In the present article we extend the above work (\cite{Gazaki2015}) in two different directions. 

\subsection{Abelian varieties over $p$-adic fields}
In the first part of the paper, we focus on the case of an abelian variety $A$ over a finite extension of the $p$-adic field $\Q_p$. In the case when $A$ has good ordinary reduction, we can say more about the quotients $F^r/F^{r+1}$ of the above filtration and about the Albanese kernel $T(A)$. In \autoref{section 2} we prove the following theorem. 

 \begin{theo}\label{main1} Let $k$ be a finite extension of the $p$-adic field $\Q_p$ with $p\geq 3$. Let $A$ be an abelian variety over $k$ with good ordinary reduction. The Albanese kernel, $T(A)$, of $A$ is the direct sum of its maximal divisible subgroup and a torsion group. 
\end{theo}
\autoref{main1} can be thought of as a promising step towards the following conjecture of Colliot-Th\'{e}l\`{e}ne (\cite{Colliot-Thelene1995}).


\begin{conj}\label{CTconj} Let $X$ be a smooth projective variety over a finite extension $k$ of $\Q_p$. The Albanese kernel $T(X)$ is a direct sum of its maximal divisible subgroup with a finite group. 
\end{conj} 

A weaker form of the above conjecture has been established by S. Saito and K. Sato (\cite{Saito/Sato2010}), who proved that the degree zero subgroup $A_0(X)$ is the direct sum of a finite group and a group  that is divisible by any integer $m$ coprime to $p$. Given that the maximal divisible subgroups of $A_0(X)$ and $T(X)$ coincide, the remaining questions include whether the quotient $T(X)/p^n$ is finite for every $n\geq 1$, and whether it stabilizes for large $n>0$. There are only very few results in this direction, including certain classes of rationally connected varieties (\cite{CT2005}) and  certain products of curves, $X=C_1\times\cdots\times C_d$ (\cite{Raskind/Spiess2000}, \cite{Gazaki/Leal2018}). 
 
 Unfortunately we are still not able to give more precise information about the structure of the torsion summand of $T(A)$ in \autoref{main1}, but we provide some indication that it could be of bounded exponent (see \autoref{exponent}). 
 \vspace{2pt}
 \subsection{Semi-abelian varieties and Suslin's singular homology}
The second goal of the paper is to extend the isomorphism \eqref{isoab} to semi-abelian varieties.  A semi-abelian variety $G$ over a field $k$ is an extension of an abelian variety by a torus; in particular, it is a quasi-projective variety. For such open varieties the study of $CH_0$ is often replaced by the study of a larger class group of zero-cycles, namely of Suslin's singular homology group, $H_0^{\sing}$. 

In \autoref{section 3}, imitating the method of \cite{Gazaki2015}, we construct, for a semi-abelian variety $G$ over a perfect field $k$, a decreasing filtration,
$F^0\supset F^1\supset\cdots\supset F^r\supset\cdots$ of $H_0^{\sing}(G)$ such that the successive quotients $F^r/F^{r+1}$ can be understood by the Somekawa K-group, $S_r(k;G)$. Namely, we prove the following theorem. 
\begin{theo}\label{isosemiab} Let $G$ be a semi-abelian variety over a perfect field $k$. There is a decreasing filtration $\{F^r\}_{r\geq 0}$ of Suslin's singular homology group, $H_0^{\sing}(G)$, such that for every $r\geq 0$ we have an isomorphism,
\[F^r/F^{r+1}\otimes\Z\left[\frac{1}{r!}\right]\simeq S_r(k;G)\otimes\Z\left[\frac{1}{r!}\right].\]  \end{theo}
 We close this introduction with the following remark.
 
\begin{rem}
We note that for a semi-abelian variety $G$, the group $H_0^{\sing}(G)\otimes\Q$ has been previously studied by Sugiyama (\cite{Sugiyama2014}), who obtained a similar isomorphism for the Pontryagin filtration $\{I^r\otimes\Q\}_{r\geq 0}$ (see \autoref{Ir}). 

Our focus both in \cite{Gazaki2015} and in the current paper is an integral study of the subject, providing some evidence towards the Beilinson conjectures. 
In fact, it was recently established by Kahn and Yamazaki (\cite{KahnYamazaki2013}) that the Somekawa K-group $K(k;\mathcal{F}_1,\ldots,\mathcal{F}_r)$ attached to more general coordinates than just semi-abelian varieties (namely homotopy invariant Nisnevich sheaves with transfers) has the expected motivic  realization.
\footnote{What they showed is an isomorphism $K(K;G_1,\ldots,G_r)\stackrel{\simeq}{\longrightarrow} 
\Hom_{\DM^{\eff}}(\Z,\mathcal{F}_1[0]\otimes\cdots\otimes \mathcal{F}_r[0])$.} 
\end{rem}


\subsection{Acknowledgements} The author is truly grateful to Professors Shuji Saito and Jean-Louis Colliot-Th\'{e}l\`{e}ne, as well as the referee for doing a careful reading of the paper and providing very helpful feedback. The author would also like to heartily thank Professors Kazuya Kato and Bhargav Bhatt and Dr. Isabel Leal for useful discussions and suggestions.


\vspace{2pt}
\section{Zero-cycles on Abelian varieties}\label{section 2} Throughout this section we will be working with an abelian variety $A$  of dimension $d$ over a field $k$. Soon enough we will focus on the case of a $p$-adic base field. 

We start by reviewing the construction of the filtration $\{F^r CH_0(A)\}_{r\geq 0}$ of $CH_0(A)$ constructed in \cite{Gazaki2015}. 
\subsection{The Somekawa K-group}
First, we consider the Somekawa K-group $K(k;A_1,\ldots,A_r)$ attached to abelian varieties $A_1,\ldots, A_r$ over $k$. This group was first defined by Somekawa in \cite{Somekawa1990}. Since then this group and many of its variants have been used by many authors, obtaining numerous applications to zero-cycles and zero-cycles with modulus (see for example, \cite{Raskind/Spiess2000}, \cite{Yamazaki2005}, \cite{KahnYamazaki2013}, \cite{Ivorra/Ruelling2017}). In this section we won't need the precise definition of $K(k;A_1,\ldots,A_r)$. We only recall the fact that it is a quotient of a simpler group, namely of the Mackey product, $(A_1\otimes^M\cdots\otimes^M A_r)(k)$. The latter is defined as follows. 
\begin{defn}\label{Mackey} Let $A_1,\ldots, A_r$ be abelian varieties over a field $k$. The Mackey product ${A_1\otimes^M\cdots\otimes^M A_r}$ is defined at a finite extension $L$ over $k$ as follows: 
\[(A_1\otimes^M\cdots\otimes^M A_r)(L):=\left(\bigoplus_{F/L\text{ finite}} A_1(F)\otimes\cdots\otimes A_r(F)\right)/R_1.\] Here $R_1$ is the subgroup generated by elements of the form, 
\[\label{projectionformula} a_1\otimes\cdots\otimes \Tr_{F'/F}(a_i)\otimes\cdots\otimes a_r-\res_{F'/F}(a_1)\otimes\cdots\otimes a_i\otimes\cdots\otimes \res_{F'/F}(a_r) \in R_1, \] 
where $F'\supset F\supset L$ is a tower of finite extensions of $k$, $a_i\in A_i(F')$ for some $i\in\{1,\cdots,r\}$, $a_j\in A_j(F)$ for every $j\neq i$, $\Tr_{F'/F}:A_i(F')\rightarrow A_i(F)$ is the trace map on abelian varieties (often referred to as \textit{the norm}) and $\res_{F'/F}:A_j(F)\hookrightarrow A_j(F')$ is the usual restriction. 
\end{defn} 
The Somekawa K-group $K(k;A_1,\ldots,A_r)$ is a quotient $(A_1\otimes^M\cdots\otimes^M A_r)(k)/R_2$, where $R_2$ is a family of relations arising from function fields of curves (see \autoref{Kgeo} for a more general definition). 
\begin{notn} We will be  using the standard notation for the generators of  $K(k;A_1,\ldots,A_r)$, namely we will write them as symbols $\{a_1,\ldots,a_r\}_{F/L}$ for $a_i\in A_i(F)$. 
\end{notn} The symmetric K-group $S_r(k;A)$ that appeared in the isomorphism \eqref{isoab} is defined as the quotient of the Somekawa K-group $K(k;A,\cdots,A)$ attached to $r$ copies of $A$ by the action of the symmetric group in $r$ variables. For $r=0$ we define $S_0(k;A)=\Z$.  
Observe that in the group $S_r(k;A)$ we have an equality, 
\[\{x_1,\ldots,x_r\}_{L/k}=\{x_{\sigma(1)},\ldots, x_{\sigma(r)}\}_{L/k},\] for every permutation $\sigma$ of $\{1,\ldots,r\}$ and for points $x_i\in A(L)$, $i=1,\ldots,r$. 
\vspace{1pt}
\subsection{A filtration of $CH_0$} We are now ready to review the construction of $F^rCH_0(A)$. In \cite[Proposition 3.1]{Gazaki2015} we defined for every $r\geq 0$ a canonical homomorphism
\[\Phi_r:CH_0(A)\rightarrow S_r(k;A),\] with $\Phi_0=\deg$. We then defined (\cite[Definition 3.2]{Gazaki2015}) the filtration $F^r$ as follows,  
\[F^0=CH_0(A),\;\;\;\text{and for }r\geq 1,\;F^rCH_0(A):=\bigcap_{j=0}^{r-1}\ker(\Phi_j).\]
It follows by the definition that for every $r\geq 0$ we have an injection, 
\[\Phi_r: F^r/F^{r+1}\hookrightarrow S_r(k;A).\]
Moreover, we showed the following properties.
\begin{itemize}
\item Keeping the notation from the introduction, the group $F^1$ coincides with the degree zero subgroup, $A_0(A)$, while the group $F^2$ coincides with the Albanese kernel, $T(A)$. 
\item For every $r\geq 0$ we defined (\cite[Proposition 3.3]{Gazaki2015}) a canonical homomorphism,
\[\Psi_r:S_r(k;A)\rightarrow F^r/F^{r+1},\] with the property $\Phi_r\circ\Psi_r=\cdot r!$. In particular, the map $\Phi_r$ becomes an isomorphism after $\displaystyle\otimes\Z\left[\frac{1}{r!}\right]$.  
\end{itemize}
\subsection*{Torsion phenomena}\label{BlochBeauville} An important property of the filtration $\{F^r\}_{r\geq 0}$ is that it contains another well known filtration, $\{I^r\}_{r\geq 0}$, previously studied independently by Bloch (\cite{Bloch1976}) and Beauville (\cite{Beauville1986}). We briefly recall how the filtration $I^r$ is defined. 

Because $A$ is an abelian variety, the group $Z_0(A)$ of zero-cycles is a group ring under the Pontryagin product. This ring structure descends to $CH_0(A)$ and  for the classes of two closed points, $a,b\in A$ the Pontryagin product is defined as follows,
\[[a]\star[b]=[a+b].\]  The filtration $\{I^r\}_{r\geq 0}$ is defined by considering the powers of the augmentation ideal, \[I:=\langle\Tr_{k'/k}([a]-[0]_{k'}):a\in A(k')\rangle,\] of the group ring. Here we denoted by $\Tr_{k'/k}:CH_0(A\otimes_k k')\rightarrow CH_0(A)$ the pushforward map induced by $\Spec(k')\rightarrow\Spec(k)$, and by $[0]_k'$ the cycle class of $0\in A\otimes_k k'$. Bloch and Beauville showed that this filtration has the property $I^r\otimes\Q=0$ for $r>d$. 

In section 4 of \cite{Gazaki2015} we showed that the filtration $\{F^r\}$ satisfies the following properties. 
\begin{itemize}
\item $F^r\supset I^r$, for every $r\geq 0$. The inclusion becomes an equality when $r=0,1$.
\item $F^r\otimes\Q=I^r\otimes\Q$, for every $r\geq 0$. 
\end{itemize}
\begin{cor}\label{alpha} The subgroup $F^{d+1}$ of $CH_0(A)$ is torsion. 
\end{cor}
\begin{proof} This follows directly from the above properties and the fact that $I^{d+1}\otimes\Q=0$. 

\end{proof}
\vspace{1pt}
\subsection{Proof of \autoref{main1}} For the remaining of this section we assume that the base field $k$ is a finite extension of $\Q_p$ with ring of integers $\mathcal{O}_k$ and residue field $\kappa$, and that $p$ is an odd prime. 
\begin{lem}\label{divisibility} Let $A$ be an abelian variety over $k$. The filtration $\{F^r\}_{r\geq 0}$ of $CH_0(A)$ has the following properties:
\begin{enumerate}
\item When $A$ has a mixture of good ordinary and split multiplicative reduction, the quotient $F^r/F^{r+1}$ is divisible for every $r\geq 3$.
\item When $A$ has good reduction, the quotient $F^2/F^3$ is the direct sum of a divisible group and a finite group. 
\end{enumerate}
\end{lem}
\begin{proof} Raskind and Spiess showed \cite[Theorem 4.5]{Raskind/Spiess2000} that for abelian varieties $A_1,\ldots,A_r$ over $k$ with a mixture of split multiplicative and good ordinary reduction, the Somekawa K-group, $K(k;A_1,\ldots,A_r)$ is divisible when $r\geq 3$ and it is the direct sum of a finite group and a divisible group when $r=2$. This implies that for $r\geq 3$, the injective homomorphism $\Phi_r:F^r/F^{r+1}\hookrightarrow S_r(k;A)$ is also surjective, because the image contains $r!S_r(k;A)$. This proves the first claim of the lemma.

Next assume that $r=2$. When $p>2$ and $A$ has good reduction, the Somekawa K-group $K(k;A,A)$ (and hence also $S_2(k;A)$) is $2$-divisible. This follows by \cite[Theorem 3.5]{Raskind/Spiess2000}. With a similar argument as above, the injection
\[\Phi_2:F^2/F^3\hookrightarrow S_2(k;A),\]  is an isomorphism.  

\end{proof}

We are now ready to prove our first theorem, which we restate here.
 \begin{theo} Let $k$ be a finite extension of the $p$-adic field $\Q_p$ with $p\geq 3$. Let $A$ be an abelian variety over $k$ with good ordinary reduction. The Albanese kernel, $T(A)$, of $A$ is the direct sum of its maximal divisible subgroup and a torsion group. 
\end{theo} 
\begin{proof}

We consider the filtration $\{F^r\}_{r\geq 0}$ of $CH_0(A)$ as above. Recall that $F^2=T(A)$ and the group $F^{d+1}$ is torsion. 

We first prove that the subgroup $F^3$ is the direct sum of a divisible group and a torsion group. Consider the short exact sequence of abelian groups,
\[0\rightarrow F^{d+1}\rightarrow F^3\rightarrow F^{3}/F^{d+1}\rightarrow 0.\] Since we assumed that the abelian variety $A$ has good ordinary reduction, part (1) of \autoref{divisibility} yields that the group $F^{3}/F^{d+1}$ has a filtration $F^{3}/F^{d+1}\supset F^{4}/F^{d+1}\supset\cdots\supset F^d/F^{d+1}\supset 0$ with each successive quotient divisible. Since divisible groups are injective $\Z$-modules, $F^{3}/F^{d+1}$ splits into a direct sum,
\[F^{3}/F^{d+1}\simeq\bigoplus_{i=3}^dF^{i}/F^{i+1},\] and it is therefore divisible. Let $\Tor(F^3)$ be the torsion subgroup of $F^3$. We have a decomposition $F^3\simeq\Tor(F^3)\oplus F^3/\Tor(F^3)$. By \autoref{alpha}, we obtain an injection $\alpha:F^{d+1}\hookrightarrow\Tor(F^3)$. We consider the commutative diagram with exact rows,
\[
	\begin{tikzcd} 
	0\ar[r]& F^{d+1}\ar{r}\ar{d}{\alpha} & F^3\ar{r}{\epsilon}\ar{d}{=} & F^3/F^{d+1}\ar{r}\ar{d}{\beta} &0\\
	0\ar{r} & \Tor(F^3)\ar{r} & F^3\ar{r}{\delta} & F^3/\Tor(F^3)\ar{r} & 0.
	\end{tikzcd}
	\] Note that the map $\beta$ is obtained by diagram chasing.  
	Because the map $\delta$ is surjective, it follows that $\beta$ is also surjective. Since $F^3/F^{d+1}$ is divisible, it follows that $F^3/\Tor(F^3)$ is divisible, which concludes the claim. 
	
We now repeat the above argument for the group $F^2$. We have a short exact sequence,
\[0\rightarrow F^3/F^{d+1}\rightarrow F^2/F^{d+1}\rightarrow F^2/F^{3}\rightarrow 0.\] Since $F^3/F^{d+1}$ is divisible, it is a direct summand of $F^2/F^{d+1}$. Part (2) of \autoref{divisibility} then yields that $F^2/F^{d+1}$ is the direct sum of a finite group and a divisible group. We will denote by $(F^2/F^{d+1})_{div}$ the maximal divisible subgroup. We consider the commutative diagram with exact rows,
\[
	\begin{tikzcd} 
	0\ar[r]& F^{d+1}\ar{r}\ar{d} & F^2\ar{r}{f}\ar{d}{=} & F^2/F^{d+1}\ar{r}\ar{d}{h} &0\\
	0\ar{r} & \Tor(F^2)\ar{r} & F^2\ar{r}{g} & F^2/\Tor(F^2)\ar{r} & 0.
	\end{tikzcd}
	\]
	By the same reasoning as above we obtain that the map $h$ is surjective. Since the group $F^2/\Tor(F^2)$ is torsion free, the finite summand of $F^2/F^{d+1}$ must map to zero under $h$ and therefore $h$ induces a surjection $(F^2/F^{d+1})_{div}\twoheadrightarrow F^2/\Tor(F^2)$. It yields that $F^2/\Tor(F^2)$ is a uniquely divisible group, which concludes the proof of the theorem.

\end{proof}
\begin{rem}\label{exponent}
Ideally we would like to at least bound the torsion of the non-divisible summand of $T(A)$. In \cite[Remark 4.6]{Gazaki2015} we suggested a possible strategy in order to bound the torsion in the group $I^{d+1}$, but we would additionally  need to control the quotient $F^{d+1}/I^{d+1}$. 
\end{rem}

\vspace{2pt}
\section{Semi-abelian Varieties and Suslin's singular homology}\label{section 3}
In this section our goal is to generalize the main result of \cite{Gazaki2015} to semi-abelian varieties. We start by reviewing the definition of Suslin's singular homology and some of its variants.  
\subsection{Suslin's singular homology} 
Let $X$ be a smooth quasi-projective variety over a field $k$. We denote by $X_{(1)}$ the set of all closed irreducible curves in $X$ and by $Z_{0}(X)$ the free abelian group of zero-cycles on $X$. Moreover, for $C\in X_{(1)}$, we denote $\widetilde{C}\stackrel{\pi}{\longrightarrow}C$ its normalization and $\iota:\widetilde{C}\hookrightarrow\overline{C}$ its smooth completion.
\begin{defn} For a smooth quasi-projective variety $X$ over a field $k$ we define Suslin's singular homology group, $H_{0}^{\sing}(X)$, as the quotient of $Z_{0}(X)$ modulo the subgroup generated by zero-cycles of the form $\iota_{0}^{\star}(Z)-\iota_{1}^{\star}(Z)$, where $\iota_{\lambda}:X\rightarrow X\times\mathbb{A}^{1}$ is the inclusion $x\rightarrow (x,\lambda)$, for $\lambda=0,1$, and $Z$ runs through all closed integral subvarieties of $X\times\mathbb{A}^{1}$ such that the projection $Z\rightarrow \mathbb{A}^{1}$ is finite and surjective.
\end{defn} The above definition occurs for the first time in the foundational paper \cite{Suslin/Voe} of Suslin and Voevodksy. This definition was later generalized by Schmidt (\cite{Schmidt2007}) to schemes of finite type over the spectrum of a Dedekind domain.

To construct our filtration we won't be working with the above definition but with a variant of this group, known as Wiesend's ideal class group. Before we review its definition, we remind the reader of the theory of generalized Jacobians for smooth complete curves, which motivated Wiesend's definition. 
\subsection*{Generalized Jacobians}
Let $C$ be a smooth curve over a perfect field $k$. Let $\overline{C}$ be its smooth compactification and let  $\displaystyle\mathfrak{m}=\sum_{P\in\overline{C}\setminus C}P$ be the reduced Cartier divisor on $\overline{C}$ supported on $\overline{C}\setminus C$. The divisor $\mathfrak{m}$ is usually called a modulus condition on $\overline{C}$. There exists a generalized Jacobian variety $J_{\mathfrak{m}}$ of $\overline{C}$ corresponding to the modulus $\mathfrak{m}$, which is a semi-abelian variety satisfying the following universal property. There is a morphism $\varphi:C\rightarrow J_{\mathfrak{m}}$ such that if $\psi:\overline{C}\dashrightarrow G$ is a rational map to some semi-abelian variety $G$ over $k$, which is regular on $C$, then  $\psi$ factors uniquely through $\varphi$. When $\overline{C}=C$, or alternatively $\mathfrak{m}=0$, $J_{\mathfrak{m}}$ coincides with the usual Jacobian variety $J$ of $\overline{C}$ and in the case $\overline{C}$ has a $k$-rational point, the Abel-Jacobi map gives an isomorphism
\[ Pic^{0}(\overline{C})\simeq \frac{Z_{0}(\overline{C})^{\deg=0}}{\langle\dv(f):f\in k(\overline{C})^{\times}\rangle}\stackrel{\simeq}{\longrightarrow}J(k).\]

When $C\subsetneq\overline{C}$ and $C(k)\neq\emptyset$, the analogous expression in terms of zero-cycles for the generalized Jacobian $J_{\mathfrak{m}}$, where $\mathfrak{m}=\sum_{P\in\overline{C}\setminus C}P$, is the following.
\[\frac{Z_{0}(C)^{\deg 0}}{\langle\dv(f):f\in k(\overline{C})^{\times}, f(P)=1, P\in \overline{C}\setminus C\rangle}\stackrel{\simeq}{\longrightarrow}J_{\mathfrak{m}}(k).\]

This is part of the more general theory of generalized Jacobians of Rosenlicht-Serre (\cite{Serre1988}). 
\vspace{2pt}
\subsection*{Wiesend's class group of zero-cycles}

\begin{defn}\label{Wi} For a smooth quasi-projective variety $X$ over a field $k$ we define the Wiesend tame ideal class group $W(X)$ to be the quotient of $Z_{0}(X)$,
by the subgroup generated by zero-cycles of the form $\{\dv(f):f\in k(C)^{\times}\}_{C\in X_{(1)}}$ with the function $f$ having the property  $f(P)=1$, for every $P\in\overline{C}-\widetilde{C}$.
\end{defn}
\begin{rem}
This definition is a direct generalization of the generalized Jacobian $J_{\mathfrak{m}}$ of a smooth complete curve corresponding to a reduced modulus $\mathfrak{m}$. We note that Wiesend's  \cite{Wiesend2007} original definition was not exactly this, but a similar class group for arithmetic schemes of finite type over $\Spec\mathbb{Z}$. \autoref{Wi} is a variant due to Yamazaki \cite{Yamazaki2013}.
\end{rem}

\subsection*{Properties}
\begin{enumerate}\item When the base field $k$ is perfect, it is a theorem of Schmidt \cite[Theorem 5.1]{Schmidt2007} that the groups $W(X)$ and $H_{0}^{\sing}(X)$ are isomorphic (see also \cite[Theorem 1]{Schmidtappendix}). From now on we assume that $k$ is perfect. In most of the statements we will be using $W(X)$, but we will be interchanging between the two definitions without further notice.
\item When $X$ is proper over $k$, the groups $CH_{0}(X)$ and $W(X)$ coincide.
\item It is clear that the degree map $\deg:Z_{0}(X)\rightarrow\mathbb{Z}$, $x\rightarrow[k(x):k]$, factors through $W(X)$. We will denote by $W^{0}(X)$ the subgroup of degree zero cycle classes.
\item When $\overline{C}$ is a smooth, complete, geometrically irreducible curve over $k$ and $S$ a finite set of points of $\overline{C}$, then for the smooth curve $C=\overline{C}-S$, the group $W(C)$ coincides with the group of classes of divisors on $\overline{C}$ prime to $S$ modulo $S$-equivalence, as defined in \cite[Chapter V]{Serre1988}. Notice that when $C$ has a $k$-rational point, the abelian group $W^{0}(C)$ is isomorphic to the generalized Jacobian of $\overline{C}$ corresponding to the modulus $\mathfrak{m}=\sum_{P\in S}P$.
\item The groups $W(X)$ and $H^{\sing}_0$ are covariant functorial for morphisms of varieties $X\rightarrow Y$ (\cite[Proposition 2.10]{Schmidt2007}, \cite[Lemma 2]{Wiesend2007}, \cite[lemma 2.3]{Yamazaki2013}).
\item \underline{Generalized Albanese map:} If $X$ is a smooth variety over a perfect field $k$, there is a generalized albanese map $\alb_{X}:W^{0}(X)\rightarrow G_{X}(k)$, where $G_{X}$ is the generalized Albanese variety of $X$. For a proof of the fact that the generalized Albanese map is well-defined we refer to \cite{Spiess/Szamuely2003}. In this article, T.Szamueli and M.Spiess prove the analogue of Roitman's theorem for Suslin's singular homology.
\end{enumerate}

\medskip
\subsection{The geometric K-group} Next we need to review the definition of the Somekawa K-group, $K(k;G_1,\ldots,G_r)$, attached to semi-abelian varieties $G_1,\ldots, G_r$ over $k$. In order to prove \autoref{isosemiab}, we will  use a more geometric variant, $K^{\geo}(k;G_1,\ldots,G_r)$, which was introduced by  Kahn and Yamazaki in \cite{KahnYamazaki2013}. We start with the following preliminary remark that will help us simplify the notation. 

\begin{rem}\label{Si} Let $K$ be a function field in one variable over $k$ and let $C$ be the smooth complete curve having function field $K$. Let $g_{i}\in G_{i}(K)$, for $i=1,\dots,r$. Each $g_{i}$ extends to a rational map $g_{i}:C\dashrightarrow G_{i}$, which is regular outside a finite set of places $S_{i}$ of $C$.  From now on we will refer to the set $S_i$ as \textit{the support of} $g_i$. Let $P$ be a closed point of $C$. If $K_{P}$ is the completion of $K$ at $P$ and $\mathcal{O}_{K_{P}}$ its ring of integers, we will denote by $\mathcal{O}_{P}$ the algebraic local ring, $K\cap\mathcal{O}_{K_{P}}$, and by $k(P)$ the residue field at $P$. Then the set $S_{i}$  is precisely the set \[S_{i}=\{P\in C:g_{i}\not\in G_{i}(\mathcal{O}_{P})\}.\] 
Moreover, Somekawa defined a \textit{tame symbol}, $\partial_P:G_i(K_P)\otimes K_P^\times\rightarrow G(k(P))$. This symbol is a direct generalization of the tame symbol of the multiplicative group, $\G_m$. We will not need the precise definition in what follows. For more details we refer to \cite[1.1]{Somekawa1990}
\end{rem}
\begin{defn}\label{Som} The Somekawa K-group $K(k;G_1,\ldots, G_r)$ attached to semi-abelian varieties $G_1, \ldots, G_r$ over a perfect field $k$ is defined as follows. \[\displaystyle K(k;G_1,\ldots, G_r):=(G_1\otimes^{M}\cdots\otimes^{M}G_r)(k)/R,\] where the subgroup $R$ is generated by the following family of elements.\\
\textbf{Weil reciprocity I:} Let $K$ be a function field in one variable over $k$ and let $C$ be the smooth complete curve having function field $K$.  Let $g_{i}\in G_i(K)$ for $i=1,\ldots, r$ be elements with disjoint supports and let $f\in K^\times$. Then for every closed point $P\in C$ there exists some $i(P)\in\{1,\ldots, r\}$ such that $P\not\in S_j$, for every $j\neq i$. We require,
    \[\sum_{P\in C}\{g_{1}(P),\ldots, \partial_{P}(g_{i(P)}\otimes f),\ldots, g_{r}(P)\}_{k(P)/k}\in R.\] 
    \end{defn}

\begin{defn}\label{Kgeo} Let $G_{1},\dots,G_{r}$ be semi-abelian varieties over a perfect field $k$. We define the geometric  K-group, $K^{\geo}(k;G_{1},\ldots,G_{r})$, attached to $G_{1},\ldots,G_{r}$ as follows. \[K^{\geo}(k;G_{1},\ldots,G_{r})=(G_{1}\otimes^{M}\dots\otimes^{M} G_{r})(k)/R_{0},\] where $R_{0}$ is the subgroup generated by the following family of elements:\\
\textbf{Weil reciprocity II:} Let $K$ be a function field in one variable over $k$ and let $C$ be the smooth complete curve having function field $K$. Let $g_{i}\in G_{i}(K)$, for $i=1,\dots,r$.  We consider the set $\displaystyle S=\bigcup_{i=1}^{r}S_{i}$. Let $f\in K^{\times}$ be a function such that $f(P)=1$, for every $P\in S$. Then we require \[\sum_{P\not\in S}\ord_{P}(f)\{g_{1}(P),\ldots, g_{r}(P)\}_{k(P)/k}\in R_{0}.\]
\end{defn} Recall that $\otimes^{M}$ is the product of Mackey functors defined in \autoref{projectionformula}. Kahn and Yamazaki proved that the group $K^{\geo}(k;G_1,\ldots,G_r)$ is isomorphic to Somekawa K-group $K(k;G_1,\ldots,G_r)$ (\cite[Theorem 11.12]{KahnYamazaki2013}), with this new variant being more suitable for geometric applications. 

\subsection{Proof of \autoref{isosemiab}}
Let $G$ be a semi-abelian variety over a perfect field $k$ of dimension $d$. We  will write the group law in $G$ in multiplicative notation with $1$ the neutral element.
\subsection*{The Pontryagin Filtration} Similarly to the case of abelian varieties, the group of zero-cycles $Z_{0}(G)$  becomes a group ring with multiplication given by the Pontryagin product, $\displaystyle(\sum_{j=1}^{s}n_{j}x_{j})\star(\sum_{i=1}^{t}m_{i}y_{i})=\sum_{i,j}n_{j}m_{i}x_{j}y_{i},$ for $x_{j},y_{i}$ closed points of $G$ and $n_{j},\;m_{i}$ integers.
\begin{lem} The subgroup $M=\langle\dv(f):f\in k(C)^{\times},C\in G_{(1)},\;f(x)=1, x\in\overline{C}-\widetilde{C}\rangle$ is an ideal of $Z_{0}(G)$ and therefore $W(G)$ becomes a ring with the Pontryagin product.
\end{lem}
\begin{proof} It suffices to show that if $x\in G$ is any closed point of $G$ and $\dv(f)$ is a generator of $M$, then $x\star\dv(f)\in M$. We consider the translation map
\begin{eqnarray*}\tau_{x}:&&G\rightarrow G\\
&&y\rightarrow xy.
\end{eqnarray*} Then we observe that $x\star\dv(f)=\tau_{x\star}(\dv(f))$ and since $W(G)$ has covariant functoriality, we conclude that $x\star\dv(f)\in M$.

\end{proof}
Under this ring structure, the subgroup of degree zero elements, $W^{0}(G)$, becomes an ideal $I$ of $W(G)$. By taking its powers, we obtain the \textit{Pontryagin filtration} $\{I^r\}_{r\geq 0}$ of $W(G)$. In what follows we will need a precise description of the generators of $I^r$, which we include below. 
\begin{defn}\label{Ir} The filtration $I^r$ of $W(G)$ is defined as follows:
\begin{eqnarray*}&&I^{0}W(G)=W(G),\\
&&I^{1}W(G)=\langle Tr_{k'/k}([x]-[1]_{k'}):x\in G(k')\rangle =W^{0}(G),\\
&&I^{2}W(G)=\langle Tr_{k'/k}([xy]-[x]-[y]+[1]_{k'}):x,\;y\in G(k')\rangle,\\
&&I^{3}W(G)=\langle Tr_{k'/k}([xyz]-[xy]-[xz]-[yz]+[x]+[y]+[z]-[1]_{k'}):x,\;y,\;z\in G(k')\rangle,\\
&&\dots\\
&&I^{r}W(G)=\langle \sum_{j=0}^{r}(-1)^{r-j}\sum_{1\leq\nu_{1}<\dots<\nu_{j}\leq r}Tr_{k'/k}([x_{\nu_{1}}x_{\nu_{2}}\dots x_{\nu_{j}}]),\;
x_{i}\in G(k')\rangle,
\end{eqnarray*} where the summand corresponding to $j=0$ is $Tr_{k'/k}((-1)^{r}[1]_{k'})$.
\end{defn}
\medskip
\begin{notn} For points $x_{i}\in G(k')$,  $i\in\{1,\ldots,r\}$,  we will denote by $\omega_{x_{1},\dots,x_{r}}$ the element $\sum_{j=0}^{r}(-1)^{r-j}\sum_{1\leq\nu_{1}<\dots<\nu_{j}\leq r}Tr_{k'/k}([x_{\nu_{1}}x_{\nu_{2}}\dots x_{\nu_{j}}])$ of $I^r$.
\end{notn}
\medskip
\begin{rem} We will see that after $\otimes\mathbb{Q}$ the filtration $\{I^{r}W(G)\}_{r\geq 0}$ just defined has the property that $I^{r}/I^{r+1}\otimes\mathbb{Q}\simeq S_{r}(k;G)\otimes\mathbb{Q}$. This is exactly what was proved also by Sugiyama in \cite[Proposition 4.8]{Sugiyama2014}.
\end{rem}
\medskip
\subsection{Definition of the Filtration}
We start this subsection with the observation that $G$ is its own generalized Albanese variety. Therefore the Albanese map takes on the form $\alb_{G}:W^{0}(G)\rightarrow G(k)$. This follows by the next proposition, which is a direct consequence of a classical result of Rosenlicht-Serre (\cite{Serre1988}).
\begin{prop}\label{recG} For any semi-abelian variety $G$ over $k$, the natural map $j:G(k)\rightarrow K_{1}^{\geo}(k;G)$ is an isomorphism.
\end{prop}
\begin{proof} First we prove surjectivity. Let $k'/k$ be a finite extension and $x\in G(k')$. Notice that in $K_{1}^{\geo}(k;G)$ we have the following equality, $\{x\}_{k'/k}=\{\Tr_{k'/k}(x)\}_{k/k}$. We conclude that $j(\Tr_{k'/k}(x))=\{x\}_{k'/k}$ and hence surjectivity follows.

To prove injectivity, it suffices to show that the Weil reciprocity relation of $K_1^{\geo}(k)$ holds already in $G(k)$. Let $K$ be a function field in one variable over $k$ and $C$ be the   smooth complete curve with function field $K$. Let $g\in G(K)$ and $f\in K^{\times}$. According to  \autoref{Si}, we obtain a regular map $C-S\rightarrow G$, where $S=\{P\in C:g\not\in G(\mathcal{O}_{P})\}$.
 \cite[Theorem 1]{Serre1988} tells us that the map $g$ admits a modulus $\mathfrak{m}$.  Moreover, \cite[Proposition 13]{Serre1988} shows that in the case $G$ is an extension of an abelian variety $A$ by $\mathbb{G}_{m}$,  $\mathfrak{m}=\sum_{P\in S}P$ is a modulus for $g$. By a simple argument using the projections of $\mathbb{G}_{m}^{\oplus d}$ to each factor, we can conclude that $\mathfrak{m}=\sum_{P\in S}P$ is a modulus for $g$ when $G$ is an arbitrary semi-abelian variety. This means that for every function $f\in K^{\times}$ such that $f(P)=1$ for every $P\in S$, it holds $\displaystyle\sum_{P\not\in S}\ord_{P}(f)\Tr_{k(P)/k}(g(P))=0$. Notice that this implies that Weil reciprocity holds in $G(k)$.

\end{proof}
The next proposition is analogous to \cite[Proposition 3.1]{Gazaki2015}.
\begin{prop}\label{phi1} For every $r\geq 1$, there exists a well-defined abelian group homomorphism
\begin{eqnarray*}\Phi_{r}:&&W(G)\rightarrow S_{r}^{\geo}(k;G)\\
&&[x]\rightarrow\{x,x,\ldots,x\}_{k(x)/k}.
\end{eqnarray*} Furthermore for $r=0$, we define $S_{0}^{\geo}(k;G)=\mathbb{Z}$ and $\Phi_{0}$ to be the degree map.
\end{prop}
\begin{proof} Let $r>0$ be a positive integer. We define the map $\Phi_{r}:Z_{0}(G)\rightarrow S_{r}^{\geo}(k;G)$ first on the level of zero-cycles. Let $C\in G_{(1)}$ be a closed irreducible curve in $G$, let $\widetilde{C}\stackrel{\pi}{\longrightarrow}C$ be its normalization and $\iota:\widetilde{C}\hookrightarrow\overline{C}$ the smooth completion of $\widetilde{C}$. Let $f\in k(\overline{C})^{\times}$ be a function such that $f(P)=1$, for every $P\in \overline{C}-\widetilde{C}$. We need to show $\Phi_{r}(\pi_{\star}(\dv(f)))=0$. More precisely, we need to prove
\[\Phi_{r}(\sum_{x\in\widetilde{C}}\ord_{x}(f)[k(x):k(\pi(x))])=
\sum_{x\in\widetilde{C}}\ord_{x}(f)[k(x):k(\pi(x))]\{\pi(x),\ldots,\pi(x)\}_{k(\pi(x))/k}=0.\]
First we have the following equalities.
\begin{eqnarray*}&&\sum_{x\in\widetilde{C}}\ord_{x}(f)[k(x):k(\pi(x))]\{\pi(x),\ldots,\pi(x)\}_{k(\pi(x))/k}=\\
&&\sum_{x\in\widetilde{C}}\ord_{x}(f)\{[k(x):k(\pi(x))]\pi(x),\ldots,\pi(x)\}_{k(\pi(x))/k}=\\
&&\sum_{x\in\widetilde{C}}\ord_{x}(f)\{Tr_{k(x)/k(\pi(x))}(\res_{k(x)/k(\pi(x))}(\pi(x))),\ldots,\pi(x)\}_{k(\pi(x))/k}=\\
&&\sum_{x\in\widetilde{C}}\ord_{x}(f)\{\res_{k(x)/k(\pi(x))}(\pi(x)),\ldots,\res_{k(x)/k(\pi(x))}(\pi(x))\}_{k(x)/k}.
\end{eqnarray*}
Let $K=k(C)$ and consider the generic point inclusion $\eta:\Spec K\hookrightarrow \widetilde{C}$. We set $g=\pi\eta\in G(K)$ and observe that $S=\{x\in\widetilde{C}:g\not\in G(\mathcal{O}_{P})\}=\overline{C}-\widetilde{C}$. Then we can easily see that
\[\sum_{x\in\widetilde{C}}\ord_{x}(f)\{\res_{k(x)/k}(\pi(x)),\dots,\res_{k(x)/k}(\pi(x))\}_{k(x)/k}=\sum_{x\not\in S}\ord_{x}(f)\{g(x),\dots,g(x)\}_{k(x)/k}.\]
The result therefore  follows from
the reciprocity relation  of the group $S_{r}^{\geo}(k;G)$.

\end{proof}
We can now proceed to the definition of the filtration. First notice that the isomorphism obtained in \autoref{recG} yields an equality $\Phi_{1}|_{\ker(\Phi_{0})}=\alb_{G}$. This in turn implies that $\ker(\Phi_{0})\cap\ker(\Phi_{1})=\ker(\alb_{G})$.
\begin{defn} We define a decreasing filtration $\{F^{r}\}_{r\geq 0}$ of $W(G)$ with $F^{0}=W(G)$ and for $r\geq 1$, $\displaystyle F^{r}=\bigcap_{j=0}^{r-1}\ker\Phi_{j}$. In particular, $F^{1}=W^{0}(G)$ and $F^{2}=\ker(\alb_{G})$.
\end{defn}
\begin{prop}\label{Gr1} For every $r\geq 0$ we have inclusions $I^{r}\subset F^{r}$. Moreover, \[\Phi_{r}(\omega_{x_{1},\dots,x_{r}})=r!\{x_{1},\dots,x_{r}\}_{k(x)/k}.\]
\end{prop}
\begin{proof}
This is analogous to \cite[Proposition 3.3, part of Proposition 3.4]{Gazaki2015}. For the inclusion $I^{r}\subset F^{r}$ we use the commutativity $\Phi_{r-1}\Tr_{k'k}=\Tr_{k'/k}\Phi_{r-1}^{k'}$ (here $\Phi_{r-1}^{k'}$ is the corresponding map defined over $\Spec(k')$), the multilinearity of the symbol in $S_{r}^{\geo}(k;G)$ and the fact that the map $\Phi_{r+1}$ is a group homomorphism. The second statement follows by a combinatorial counting.

\end{proof}
\begin{thm}\label{psi1} Let $r\geq 0$ be an integer. There is a well-defined abelian group homomorphism
\begin{eqnarray*}&&\Psi_{r}:S_{r}^{\geo}(k;G)\longrightarrow \frac{F^{r}W(G)}{F^{r+1}W(G)}\\
&&\{x_{1},\dots,x_{r}\}_{k'/k}\longrightarrow\omega_{x_{1},\dots,x_{r}}.
\end{eqnarray*} Moreover, the homomorphism $\Psi_{r}$ satisfies the property, $\Phi_{r}\circ\Psi_{r}=\cdot r!$ on $S_{r}(k;G)$. As a conclusion, after $\displaystyle\otimes\mathbb{Z}\left[\frac{1}{r!}\right]$ the map $\Phi_{r}$ becomes an isomorphism with inverse $\displaystyle\frac{1}{r!}\Psi_{r}$.
\end{thm}
\begin{proof} The first step is to obtain a well-defined map, for every $r\geq 0$,
\[\Psi_{r}:\frac{(\overbrace{G\otimes^{M}\cdots\otimes^{M}G}^{r})(k)}{<(x_{1}\otimes\ldots\otimes x_{r})_{k'/k}-
(x_{\sigma(1)}\otimes\ldots\otimes x_{\sigma(r)})_{k'/k}>}\longrightarrow \frac{F^{r}W(G)}{F^{r+1}W(G)}.\]  The argument is exactly the same as  the first two steps of \cite[Proposition 3.4]{Gazaki2015}.

We will now show that this map factors through $S_{r}^{\geo}(k;G)$. Let $C$ be a smooth complete curve over $k$ having function field $K$. Let $g\in G(K)$ and $S=\{P\in C:g\not\in G(\mathcal{O}_{P})\}$ be the support of $g$. Let $f\in K^{\times}$ be a function such that $f(P)=1$, for every $P\in S$. We need to show that \[\sum_{P\not\in S}\ord_{P}(f)\Tr_{k(P)/k}([g(P)])=0\] in $W(G)$.  Set $C_{0}=C-S$. Then $g$ induces a morphism
$g:C_{0}\rightarrow G$.  Since Wiesend's ideal class group is covariant functorial, we obtain a push forward
$g_{\star}:W(C_{0})\rightarrow W(G).$ By property (4) of Wiesend's class group, we have that the group $W(C_{0})$ is equal to the group of divisors on $C$ prime to $S$ modulo $S$-equivalence, and therefore $\dv(f)=0$ in $W(C_{0})$. This forces  \[g_{\star}(\dv(f))=\sum_{P\not\in S}\ord_{P}(f)\Tr_{k(P)/k}([g(P)])=0\in W(G).\]

\end{proof}

The above proposition concludes the proof of \autoref{isosemiab}. 
\medskip
\begin{rem} When the base field $k$ is algebraically closed \autoref{isosemiab} holds integrally, 
\[F^rW(G)/F^{r+1}W(G)\simeq S_r^{\geo}(k;G).\] For, the group $S_r^{\geo}(k;G)$ is divisible in this case, and hence the map $\Phi_r$ is surjective. 
\end{rem}

\vspace{10pt}

\bibliographystyle{amsalpha}

\bibliography{bibfile}

\providecommand{\bysame}{\leavevmode\hbox to3em{\hrulefill}\thinspace}
\providecommand{\MR}{\relax\ifhmode\unskip\space\fi MR }
\providecommand{\MRhref}[2]{%
  \href{http://www.ams.org/mathscinet-getitem?mr=#1}{#2}
}
\providecommand{\href}[2]{#2}
\begin{thebibliography}{Sch07b}

\bibitem[Bea86]{Beauville1986}
Arnaud Beauville, \emph{Sur l'anneau de {C}how d'une vari\'et\'e ab\'elienne},
  Math. Ann. \textbf{273} (1986), no.~4, 647--651. \MR{826463 (87g:14049)}

\bibitem[Bl84]{Beilinson1984}
A.~A. Be\u\i~linson, \emph{Higher regulators and values of {$L$}-functions},
  Current problems in mathematics, {V}ol. 24, Itogi Nauki i Tekhniki, Akad.
  Nauk SSSR, Vsesoyuz. Inst. Nauchn. i Tekhn. Inform., Moscow, 1984,
  pp.~181--238. \MR{760999}

\bibitem[Blo76]{Bloch1976}
Spencer Bloch, \emph{Some elementary theorems about algebraic cycles on
  {A}belian varieties}, Invent. Math. \textbf{37} (1976), no.~3, 215--228.
  \MR{0429883 (55 \#2892)}

\bibitem[Blo00]{Bloch2000}
Spencer~J. Bloch, \emph{Higher regulators, algebraic {$K$}-theory, and zeta
  functions of elliptic curves}, CRM Monograph Series, vol.~11, American
  Mathematical Society, Providence, RI, 2000. \MR{1760901}

\bibitem[CT95]{Colliot-Thelene1995}
Jean-Louis Colliot-Th\'el\`ene, \emph{L'arithm\'etique du groupe de {C}how des
  z\'ero-cycles}, J. Th\'eor. Nombres Bordeaux \textbf{7} (1995), no.~1,
  51--73, Les Dix-huiti\`emes Journ\'ees Arithm\'etiques (Bordeaux, 1993).
  \MR{1413566}

\bibitem[CT05]{CT2005}
\bysame, \emph{Un th\'eor\`eme de finitude pour le groupe de {C}how des
  z\'ero-cycles d'un groupe alg\'ebrique lin\'eaire sur un corps {$p$}-adique},
  Invent. Math. \textbf{159} (2005), no.~3, 589--606. \MR{2125734}

\bibitem[Gaz15]{Gazaki2015}
Evangelia Gazaki, \emph{On a filtration of {$CH_0$} for an abelian variety},
  Compos. Math. \textbf{151} (2015), no.~3, 435--460. \MR{3320568}

\bibitem[GL18]{Gazaki/Leal2018}
E.~{Gazaki} and I.~{Leal}, \emph{{Zero-cycles on a product of elliptic curves
  over a $p$-adic field}}, ArXiv e-prints (2018).

\bibitem[IR17]{Ivorra/Ruelling2017}
Florian Ivorra and Kay R\"ulling, \emph{K-groups of reciprocity functors}, J.
  Algebraic Geom. \textbf{26} (2017), no.~2, 199--278. \MR{3606996}

\bibitem[KY13]{KahnYamazaki2013}
Bruno Kahn and Takao Yamazaki, \emph{Voevodsky's motives and {W}eil
  reciprocity}, Duke Math. J. \textbf{162} (2013), no.~14, 2751--2796.
  \MR{3127813}

\bibitem[RS00]{Raskind/Spiess2000}
Wayne Raskind and Michael Spiess, \emph{Milnor {$K$}-groups and zero-cycles on
  products of curves over p-adic fields}, Compositio Math. \textbf{121} (2000),
  no.~1, 1--33. \MR{1753108 (2002b:14007)}

\bibitem[Sch07a]{Schmidt2007}
Alexander Schmidt, \emph{Singular homology of arithmetic schemes}, Algebra
  Number Theory \textbf{1} (2007), no.~2, 183--222. \MR{2361940}

\bibitem[Sch07b]{Schmidtappendix}
\bysame, \emph{Some consequences of {W}iesend's higher dimensional class field
  theory. {A}ppendix to: ``{C}lass field theory for arithmetic schemes''
  [{M}ath. {Z}. {\bf 256} (2007), no. 4, 717--729; mr2308885] by {G}.
  {W}iesend}, Math. Z. \textbf{256} (2007), no.~4, 731--736. \MR{2308886}

\bibitem[Ser88]{Serre1988}
Jean-Pierre Serre, \emph{Algebraic groups and class fields}, Graduate Texts in
  Mathematics, vol. 117, Springer-Verlag, New York, 1988, Translated from the
  French. \MR{918564}

\bibitem[Som90]{Somekawa1990}
M.~Somekawa, \emph{On {M}ilnor {$K$}-groups attached to semi-abelian
  varieties}, $K$-Theory \textbf{4} (1990), no.~2, 105--119. \MR{1081654
  (91k:11052)}

\bibitem[SS03]{Spiess/Szamuely2003}
Michael Spiess and Tam\'as Szamuely, \emph{On the {A}lbanese map for smooth
  quasi-projective varieties}, Math. Ann. \textbf{325} (2003), no.~1, 1--17.
  \MR{1957261}

\bibitem[SS10]{Saito/Sato2010}
Shuji Saito and Kanetomo Sato, \emph{A finiteness theorem for zero-cycles over
  {$p$}-adic fields}, Ann. of Math. (2) \textbf{172} (2010), no.~3, 1593--1639,
  With an appendix by Uwe Jannsen. \MR{2726095}

\bibitem[Sug14]{Sugiyama2014}
Rin Sugiyama, \emph{Motivic homology of semiabelian varieties}, Doc. Math.
  \textbf{19} (2014), 1061--1084. \MR{3272920}

\bibitem[SV96]{Suslin/Voe}
Andrei Suslin and Vladimir Voevodsky, \emph{Singular homology of abstract
  algebraic varieties}, Invent. Math. \textbf{123} (1996), no.~1, 61--94.
  \MR{1376246}

\bibitem[Wie07]{Wiesend2007}
G{\"o}tz Wiesend, \emph{Class field theory for arithmetic schemes}, Math. Z.
  \textbf{256} (2007), no.~4, 717--729. \MR{2308885}

\bibitem[Yam05]{Yamazaki2005}
Takao Yamazaki, \emph{On {C}how and {B}rauer groups of a product of {M}umford
  curves}, Math. Ann. \textbf{333} (2005), no.~3, 549--567. \MR{2198799}

\bibitem[Yam13]{Yamazaki2013}
\bysame, \emph{The {B}rauer-{M}anin pairing, class field theory, and motivic
  homology}, Nagoya Math. J. \textbf{210} (2013), 29--58. \MR{3079274}

\end{thebibliography}

\end{document}